\documentclass[a4paper,11pt]{article}
\usepackage{epsfig,amsmath}
\usepackage[french,english]{babel}
\usepackage[T1]{fontenc}
\usepackage{amsthm}
\usepackage{amssymb}
\usepackage{multicol}
\usepackage{enumerate}
\usepackage[latin1]{inputenc}
\usepackage{pstricks,pst-node,pst-text,pst-plot,pst-poly}
\newcommand{\ds}{\displaystyle}

\renewcommand{\leq}{\leqslant}
\renewcommand{\geq}{\geqslant}

\theoremstyle{definition}

\newtheorem{theorem}{Theorem}[section]
\newtheorem{lemma}[theorem]{Lemma}
\newtheorem{sublemma}[theorem]{Sublemma}
\numberwithin{equation}{section}
\date{June 2013}    
\begin{document}
\def\og{\leavevmode\raise.3ex\hbox{$\scriptscriptstyle\langle\!\langle$~}}
\def\fg{\leavevmode\raise.3ex\hbox{~$\!\scriptscriptstyle\,\rangle\!\rangle$}}
%
\title{Lower bound for the maximum of some derivative of Hardy's~function}
\author{Philippe Blanc\footnote{E-mail: {\tt philippe.blanc@heig-vd.ch}}}

\maketitle 
\begin{center}
{Haute École d'Ingénierie et de Gestion\\ CH-1400 Yverdon-les-Bains\\}
\end{center}

\noindent
{\bf Keywords} : Riemann zeta function, Distribution of zeros, Hardy's function.

\vspace*{0.25cm}

\begin{abstract}
\noindent  Under the Riemann hypothesis, we use the distribution of zeros of the zeta function to get a lower bound for the maximum of some derivative of Hardy's function.
\end{abstract}
\section{Introduction and main results} 
\label{intro}
 To situate the problem we address here, we recall some classical results on the zeros of the Riemann zeta function.\\
We denote as usual by $Z$ the Hardy function whose real zeros coincide with the zeros of $\zeta $ located on the line of real part $\frac{1}{2}$. If the Riemann hypothesis is true, what we assume from now on, then the number of zeros of $Z$ in the interval $ ] 0, t] $ is given by [9]
\begin{equation} \label{eq:N}
N(t)=\frac{t}{2 \pi}\log \left( \frac{t}{2 \pi e}\right) +\frac{7}{8}+S(t)+ O \left( \frac {1}{t}\right)
\end{equation}
where $\displaystyle S(t)=\frac{1}{\pi}\arg \zeta(\frac{1}{2}+it)$ if $t$ is not a zero of $Z$ and $\arg \zeta(\frac{1}{2}+it)$ is defined by continuous variation along the straight lines joining $2$, $2+it$ and $\frac{1}{2}+it$ starting with the initial value $\arg \zeta(2)=0$. If $t$ is a zero of $Z$ we set $S(t)=\lim_{\epsilon \to 0_+}S(t+\epsilon)$. It is well known that
\[
\displaystyle  S(t)=\displaystyle O\Big{(}\,\frac{\log t }{\log \log t}\,\Big{)}\,.
\]
Given $T$ such that $Z(T)\neq 0$ we denote by $\gamma_k$ the real zeros of Hardy's function numbered so that $\ldots\leq\gamma_{-2}\leq\gamma_{-1}<T<\gamma_{1}\leq \gamma_{2}\leq \ldots$ Using the bound [3]
\[
\vert S(t+h)-S(t)\vert \leq  \displaystyle \Big{(}\frac{1}{2}+o(1)\Big{)}\,\frac{\log t }{\log \log t}
\]
for $0<h\leq \sqrt{t}$ where $t$ is sufficiently large, we infer from (\ref{eq:N}) that
\begin{equation} \label{eq:gamma}
\mid \gamma_{\pm k}-T \mid \leq (k-1)\frac{\pi}{\log \sqrt{\frac{T}{2\pi}}}+\Big{(}\pi + o(1)\Big{)}\frac{1}{\log\log T}\hspace{3mm}\mbox{for }k=1,2,\cdots,l
\end{equation}
where $\displaystyle l=\lfloor \sqrt{T}\rfloor$ and $T$ is large. As a consequence of this relation and without using any other properties of the zeta function we prove :
\begin{theorem}\label{mainth}
If the Riemann hypothesis holds, then for any fixed $ C > (2\log 2)^{-1}$ and any sufficiently large~$T$, there exists $k\in\{1,3,5,\ldots,2m-1,2m\}$, where $m= \lfloor C \log T\log\log T\rfloor$, such that
\[
\max_{t\in [T-2\pi,T+2\pi]}\vert Z^{(k)}(t)\vert\geq \,\Big{(}1-\frac{\log\log\log T}{\log\log T}\Big{)}^k \Big{(}\log \sqrt{\frac{T}{2\pi}}\;\Big{)}^k\,\vert Z(T)\vert\,.
\]
\end{theorem}
\noindent Theorem \ref{mainth} is a consequence of the following result.
\begin{theorem}\label{genth}
 Let $f \in C^{\infty}(\mathbb{R})$ satisfying $f(0)=1$ and vanishing at $x_k$ where the $x_k$ are numbered taking into account their multiplicity and $x_{-n}\leq \ldots\leq x_{-1}<0<x_{1}\leq \ldots \leq x_n$ and let $s$ such that
\[
\vert x_{\pm k}\vert \leq (k-1)\pi +s \hspace{3mm}\mbox{for}\hspace{3mm}k=1, 2,\ldots,n.
\]
We assume that there exists a constant $0<c<1$ such that $-a<x_{-n}\leq \ldots\leq x_{n}<a$ where $a=(n-\frac{1}{2})\frac{\pi}{c}$ and 
\[
\vert f^{(2j-1)}(\pm a)\vert \leq c^{2j-1} \hspace{3mm}\mbox{for}\hspace{3mm} j=1,\ldots,m \hspace{3 mm} \mbox{and}\hspace{3mm}\max_{x\in[-a,a]}\vert f^{(2m)}(x)\vert\leq c^{2m}
\]
for some integer $m\geq n\log n$. Then for any sufficiently large $n$ and any $0<\epsilon < \log 2$ there exists $0<c_{\epsilon}<1$ depending only on $\epsilon$ such that if $c_{\epsilon}<c<1-\frac{1}{2n}$ then
\begin{equation} \label{eq:lower}
s\geq (\log 2 - \epsilon)\,\frac{1-c}{\vert \log (1-c)\vert}\,n\,\pi\,.
\end{equation}
\end{theorem}

 This work stems from an observation of A.Ivi\`c [6] about the values of the derivatives of $Z$ in a neighborood of points where  $\vert Z \vert$ reaches a large value. In [1] we made a first step toward the proof of Theorem \ref{genth} by solving a simpler problem of the same nature.\\ 
The organization of this paper is as follows : In Section 2 we prove the key identity, a property of the derivatives of Bernoulli polynomials and preparatory lemmas. The proofs of Theorem~\ref{mainth} and \ref{genth} are given in Section 3.\\
The notations used in this paper are standard : we denote by $\lfloor x\rfloor$ the usual floor function and we set
$\{x\}=x-\lfloor x\rfloor$. As usual $B_n(x)$ and $T_n(x)$ stand for Bernoulli and Chebyshev polynomial of degree $n$.

\section{Preliminary results}
 We first prove an identity which will be used later to establish a relation between the value of a function $f\in C^{2m}[-a,a]$ at $0$, the zeros of $f$ and the values of its derivatives of odd order on the boundaries of the interval.
\begin{lemma}\label{1}
Let $-a<x_{-n}<\ldots<x_{-1}<x_0<x_1<\ldots<x_n<a$ and for $l=1,2,\ldots$ let
$\Psi_{2l-1}$ be the function defined on $[-a,a]$ by
\[
 \Psi_{2l-1}(x)=\frac{(4a)^{2l-1}}{(2l)!}\,\sum_{k=-n}^{n}\,\mu_k\,\Big{(}B_{2l}\big{(}\,\frac{1}{2}+\frac{x+x_k}{4a}\,\big{)}+B_{2l}\big{(}\,\Big{\{}\frac{x-x_k}{4a}\Big{\}}\,\big{)}\Big{)}
 \]
where $\displaystyle \sum_{k=-n}^{n}\mu_k =0$.
Then for $f\in C^{2m}[-a,a]$ where $m\geq 1$ we have the identity
\begin{equation} \label{eq:key}
\sum_{k=-n}^n\,\mu_k f(x_k)=\sum_{k=1}^{m}f^{(2k-1)}(a)\Psi_{2k-1}(a)-\sum_{k=1}^{m}f^{(2k-1)}(-a)\Psi_{2k-1}(-a)-\int\limits_{-a}^{a}f^{(2m)}(x)\Psi_{2m-1}(x)\, dx.
\end{equation}
\end{lemma}
\begin{proof}
By definition the function $\Psi_{2m-1}$ is $C^{2m-2}$, piecewise polynomial and the relation $B_l'(x)~=~lB_{l-1}(x)$ for $l=1,2,\ldots$ leads to
\[
 \Psi_{2m-1}^{(j)}(x)=\frac{(4a)^{2m-j-1}}{(2m-j)!}\,\sum_{k=-n}^{n}\,\mu_k\,\Big{(}B_{2m-j}\big{(}\,\frac{1}{2}+\frac{x+x_k}{4a}\,\big{)}+B_{2m-j}\big{(}\,\Big{\{}\frac{x-x_k}{4a}\Big{\}}\,\big{)}\Big{)}
 \]
for $j=1,\ldots,2m-1$ and $x\neq x_k$ if $j=2m-1$.
This implies that
\begin{equation}\label{eq:psipaire} 
\Psi_{2m-1}^{(2m-2j)}=\Psi_{2j-1}\hspace{0.2cm}\mbox{for}\hspace{0.2cm} j=1,2,\ldots,m
\end{equation}
and that
\begin{eqnarray}\label{eq:psiimpaire}
\Psi_{2m-1}^{(2m-2j+1)}(\pm a)&=&\frac{(4a)^{2j-2}}{(2j-1)!}\,\sum_{k=-n}^{n}\,\mu_k\,\Big{(}B_{2j-1}\big{(}\,\frac{1}{2}+\frac{\pm a+x_k}{4a}\,\big{)}+B_{2j-1}\big{(}\,\Big{\{}\frac{\pm a-x_k}{4a}\Big{\}}\,\big{)}\Big{)}\nonumber\\
&=&\frac{(4a)^{2j-2}}{(2j-1)!}\,\sum_{k=-n}^{n}\,\mu_k\,\Big{(}B_{2j-1}\big{(}\,\frac{1}{2}+\frac{\pm a+x_k}{4a}\,\big{)}+B_{2j-1}\big{(}\,\frac{1}{2}-\frac{\pm a +x_k}{4a}\,\big{)}\Big{)}\nonumber\\
&=&0\hspace{0.2cm}\mbox{for}\hspace{0.2cm} j=1,2,\ldots,m
\end{eqnarray}
since $B_{2l-1}(\frac{1}{2}+x)=-B_{2l-1}(\frac{1}{2}-x)$ for $l=1,2,\ldots$
Further for $x\neq x_k$ we have
\begin{eqnarray}\label{eq:derivative}
 \Psi_{2m-1}^{(2m-1)}(x) &=&\sum_{k=-n}^{n}\,\mu_k\,\Big{(}B_{1}\big{(}\,\frac{1}{2}+\frac{x+x_k}{4a}\,\big{)}+B_{1}\big{(}\,\Big{\{}\frac{x-x_k}{4a}\Big{\}}\,\big{)}\Big{)}\nonumber\\
 &=&
 \sum_{k=-n}^{n}\,\mu_k\,\Big{(}\,\frac{x+x_k}{4a}+\Big{\{}\frac{x-x_k}{4a}\Big{\}}-\frac{1}{2}\,\Big{)}
\end{eqnarray}
and as $\displaystyle \sum_{k=-n}^{n}\mu_k =0$ the function $\Psi_{2m-1}^{(2m-1)}$ is piecewise constant. Explicitly, for $x\in ]x_j,x_{j+1}[$ we get
\[
\Psi_{2m-1}^{(2m-1)}(x)=\sum_{k=-n}^j\,\mu_k\,(\,\frac{x}{2a}-\frac{1}{2}\,)+ \sum_{k=j+1}^n\,\mu_k\,(\,\frac{x}{2a}+1-\frac{1}{2}\,)=\sum_{k=j+1}^n \mu_k=-\sum_{k=-n}^j \mu_k
\]
which leads to 
\[
\int\limits_{x_{j}}^{x_{j+1}}f'(x)\Psi_{2m-1}^{(2m-1)}(x)\,dx=-\Big{(}\sum_{k=-n}^{j}\mu_k\Big{)}(f(x_{j+1})-f(x_j)) \hspace{0.2cm} \mbox{for}\hspace{0.2 cm} j=-n,\ldots,n-1.
\]
Summing these equalities and using that $\Psi_{2m-1}^{(2m-1)}=0$ on the intervals $[-a,x_{-n}[$ and $]x_n,a]$, which follows from (\ref{eq:derivative}),  we have
\[
\sum_{k=-n}^n\,\mu_k f(x_k)=\int\limits_{-a}^{a}f'(x)\Psi_{2m-1}^{(2m-1)}(x)\,dx
\]
and we complete the proof by integrating $2m-1$ times the right-hand side by parts taking into account relations (\ref{eq:psipaire}) and (\ref{eq:psiimpaire}).
\end{proof}
For further use we recall some elementary facts concerning the divided differences.
\begin{lemma}\label{2}
Let $f\in C^{2n}]\negmedspace -T,T[$ and let $g$ be the function defined for pairwise distinct numbers $t_{-n},\ldots,t_n\in ]\negmedspace -T,T[$ by
\[
g(t_{-n},\ldots,t_{n})=\sum_{k=-n}^{n}\frac{f(t_k)}{\displaystyle \prod_{\genfrac{}{}{0 cm}{} {-n\leq j\leq n}{\; j\neq k}}(t_k-t_j)}\,.
\]
\newpage
Then
\begin{enumerate}[\indent]
\item[a)] The function $g$ has a continuous extension $g^*$ defined for $t_{-n},\ldots, t_n\in ]\negmedspace -T,T[$.
\item[b)] There exist $\eta=\eta(t_{-n},\ldots,t_n)$ and $\xi=\xi(t_{-n},\ldots,t_n) \in ]\negmedspace-T,T[$ such that
\[ 
g^*(t_{-n},\ldots,t_n)=\frac{f^{(2n)}(\eta)}{(2n)!}\hspace{3mm}\mbox{and}\hspace{3mm}\frac{\partial}{\partial t_i}\Big{(}(t_i-t_j)g^*(t_{-n},\ldots,t_n)\Big{)}=\frac{f^{(2n)}(\xi)}{(2n)!}\hspace{3mm}\mbox{if}\hspace{3mm}i\neq j.
\]
Moreover if $f\in C^{2n+1}]-T,T[$ there exists $\tau=\tau(t_{-n},\ldots,t_n)\in ]\negmedspace-T,T[$ such that
\[
\frac{\partial}{\partial t_i}g^*(t_{-n},\ldots,t_n)=\frac{f^{(2n+1)}(\tau)}{(2n+1)!}\,.
\] 
\item[c)]
Let $y_0,y_1,\ldots,y_l$ be the distinct values of $t_{-n},\ldots,t_n$ considered as fixed and let $r_k$ be the number of index $j$ such that $t_j=y_k$. Then there exist $\alpha_{k,i}$ depending on $y_0,y_1,\ldots,y_l$ such that
\[
g^*(t_{-n},\ldots,t_{n})=\sum_{k=0}^l \sum_{i=0}^{r_k - 1}\alpha_{k,i}f^{(i)}(y_k).
\]
\end{enumerate}
\end{lemma}
\begin{proof}
{ ~  }
\begin{enumerate}[\indent]
\item[a)]
This is a consequence of the representation formula 
\begin{equation}\label{eq:rel}
g(t_{-n},\ldots,t_{n})=
\int\limits_0^1\,d\tau_1\,\int\limits_0^{\tau_1}\,d\tau_2\cdots\int\limits_0^{\tau_{2n-1}}\,
f^{(2n)}\,(t_{-n}+\sum_{k=1}^{2n}\,\tau_k(t_{-n+k}-t_{-n+k-1}))\, d\tau_{2n}\,.
\end{equation}
\item[b)]
The first and last assertions follow from relation (\ref{eq:rel}) together with the mean value theorem. 
Since divided differences are invariant by permutation it is sufficient to prove the second assertion for $i=n$ and $j=n-1$. Multiplying (\ref{eq:rel}) by $t_n-t_{n-1}$ and integrating with respect to $\tau_{2n}$ we have
\begin{eqnarray*}
&\displaystyle (t_n-t_{n-1})g^*(t_{-n},\ldots,t_{n}) =\\
&\displaystyle \int\limits_0^1\,d\tau_1\cdots\int\limits_0^{\tau_{2n-2}}
f^{(2n-1)}(t_{-n}+\sum_{k=1}^{2n-1}\tau_k(t_{-n+k}-t_{-n+k-1})+\tau_{2n-1}(t_{n}-t_{n-1}))d\tau_{2n-1}-\\
&\displaystyle \int\limits_0^1\,d\tau_1\cdots\int\limits_0^{\tau_{2n-2}}
f^{(2n-1)}(t_{-n}+\sum_{k=1}^{2n-1}\tau_k(t_{-n+k}-t_{-n+k-1}))d\tau_{2n-1}
\end{eqnarray*}
and therefore
\begin{eqnarray*}
&\displaystyle\frac{\partial}{\partial t_n}\Big{(}(t_n-t_{n-1})g^*(t_{-n},\ldots,t_{n})\Big{)} =\\
&\displaystyle\int\limits_0^1\,d\tau_1\cdots\int\limits_0^{\tau_{2n-2}}\,
f^{(2n)}\,(t_{-n}+\sum_{k=1}^{2n-1}\,\tau_k(t_{-n+k}-t_{-n+k-1})+\tau_{2n-1}(t_{n}-t_{n-1}))\tau_{2n-1}\, d\tau_{2n-1}\,.
\end{eqnarray*} 
The use of the mean value theorem completes the proof.
\item[c)]
The proof is given in [7].
\end{enumerate}
\end{proof}
In the next lemma we indicate the choice of coefficients $\mu_k$ for which the identity of Lemma~\ref{1} is of practical use for large values of $a$. The main reason of this choice will appear in the proof of c) of Lemma~\ref{5}.
\begin{lemma}\label{3}
Let $\omega$ and $\Omega$ be the sets defined by
\[
\omega =\{(x_{-n},\ldots,x_{-1},x_1,\ldots,x_n)\in \mathbb{R}^{2n}\mid -a<x_{-n}<\ldots<x_{-1}<0<x_1<\ldots<x_n<a \}
 \]
  and 
  \[\Omega=\{(x_{-n},\ldots,x_{-1},x_1,\ldots,x_n)\in \mathbb{R}^{2n}\mid -a<x_{-n}\leq\ldots\leq x_{-1}< 0 < x_1\leq\ldots\leq x_n<a \}
  \]
   and let further $\Psi_{2l-1}$ be the function defined on $\omega \times[-a,a]$ by 
\[
\Psi_{2l-1}(x_{-n},\ldots,x_{-1},x_1,\ldots,x_n,x)=\frac{(4a)^{2l-1}}{(2l)!}\,\sum_{k=-n}^{n}\,\mu_k\,\Big{(}B_{2l}\big{(}\frac{1}{2}+\frac{x+x_k}{4a}\big{)}+B_{2l}\big{(}\Big{\{}\frac{x-x_k}{4a}\Big{\}}\big{)}\Big{)}
\]
where $x_0=0$, $\displaystyle \mu_k=\frac{\alpha_k}{\alpha_0}$ and 
\[
\alpha_k=\frac{\rule[-2 mm]{0 cm}{4 mm}1}{\rule[3 mm]{0 cm}{4 mm}\displaystyle \prod_{\genfrac{}{}{0 cm}{} {-n\leq j\leq n}{\; j\neq k}}\Big{(}\sin(\pi \frac{x_k}{2a})-\sin(\pi \frac{x_j}{2a})\Big{)}}\hspace{0.5cm} \mbox{for }k=-n,\ldots,n.\\
\]
Then
\begin{enumerate}[\indent]
\item[a)] For $l\geq 1$ the functions $\Psi_{2l-1}(\cdot,\ldots,\cdot,\pm a)$ have continuous extensions $\Psi_{2l-1}^*(\cdot,\ldots,\cdot,\pm a)$ to $\Omega$.  
\item[b)] For $l\geq n+1$ the function $\Psi_{2l-1}$ has a continuous extension $\Psi_{2l-1}^*$ to $\Omega\times [-a,a]$.
\item[c)] Let $m\geq n+1$ and $f\in C^{2m}[-a,a]$ a function defined on $[-a,a]$ which vanishes at $x_k$ where $a<x_{-n}\leq\ldots\leq x_{-1}< 0< x_1\leq\ldots\leq x_n<a$ and the $x_k$ are numbered taking into account their multiplicity. Then we have the identity 
\begin{equation}\label{eq:main}
f(0)=\sum_{k=1}^{m}f^{(2k-1)}(a)\Psi_{2k-1}^*(a)-\sum_{k=1}^{m}f^{(2k-1)}(-a)\Psi_{2k-1}^*(-a)-
\int\limits_{-a}^{a}f^{(2m)}(x)\Psi_{2m-1}^*(x)\, dx
\end{equation}
where for short $\Psi_{2k-1}^*(\pm a)$ and $\Psi_{2m-1}^*(x)$ stand for $\Psi_{2k-1}^*(\cdot,\ldots,\cdot,\pm a)$ and $\Psi_{2m-1}^*(\cdot,\ldots,\cdot,x)$\,.
\end{enumerate}
\end{lemma}
\begin{proof}
{~}
\begin{enumerate}[\indent]
\item[a)]Introducing the function $h$ defined by
\[
h(t,x)=\frac{(4a)^{2l-1}}{(2l)!}\Big{(}B_{2l}\big{(}\frac{1}{2}+\frac{x}{4a}+\frac{1}{2\pi}\mbox{Arcsin}\,t\big{)}+B_{2l}\big{(}\Big{\{} \frac{x}{4a}-\frac{1}{2\pi}\mbox{Arcsin}\,t\Big{\}}\big{)}\Big{)}
\]
we have
\[ 
\displaystyle\Psi_{2l-1}(x_{-n},\ldots,x_n,\pm a)=\frac{1}{\alpha_0}\,\sum_{k=-n}^n \alpha_k\, h(\sin(\pi\frac{x_k}{2a}),\pm a)
\]
for $(x_{-n},\ldots,x_n)\in \omega$ and the conclusion holds since the functions $h(\cdot,\pm a)$ belong to $C^{\infty}]\negmedspace -1,1[$\,.
\item[b)]
By definition the function $h$ belongs to $C^{2l-2}\big{(} ]\negmedspace -1,1[\times[-a,a]\big{)}$ and the assertion is a consequence of the representation formula (\ref{eq:rel}) since $2l-2\geq 2n$. 
\item[c)]
For $(x_{-n},\ldots,x_n)\in \omega$ the left-hand side of identity (\ref{eq:key}) writes 
 \[
\frac{1}{\alpha_0}\sum_{k=-n}^n \alpha_k f(\frac{2a}{\pi} \mbox{Arcsin}(\,\sin(\pi \frac{x_k}{2a})))
 \]
and thanks to Lemma \ref{2} this expression and hence the identity (\ref{eq:key}) extend to $(x_{-n},\ldots,x_n)\in \Omega$. One completes the proof by observing, thanks 
to Lemma \ref{2}, that the left-hand side reduces to $f(0)$ when the $x_k$ are zeros of multiplicity $r_k$ of $f$.
\end{enumerate}
\end{proof}
\newpage
The results stated in Lemma \ref{4} play a central role in the proof of main properties of functions $\Psi_{2l-1}^*(\cdot,\ldots,\cdot,\pm a)$.
\begin{lemma}\label{4}
For all $m,k \in \mathbb{N^*}$ we have the inequality
\[
(-1)^{m+1}\frac{d^k}{dx^k}B_{2m}(\,\frac{1}{2}+\frac{1}{\pi}\mbox{\,Arcsin}\sqrt{x}\;)>0
\hspace{0.5 cm}\mbox{for }x\in [0,1[\,.
\]
\end{lemma}
The proof of Lemma \ref{4} requires two technical results given in Sublemmas \ref{41} and \ref{42}.
\begin{sublemma}\label{41}
For all $k\in \mathbb{N}$ we have the Taylor expansion
\[
(\mbox{Arcsin}\, x)^{2k}=\sum_{l=0}^{\infty }\frac{(2k)!}{(2l)!}2^{2l-2k}\,b_{k\,,\,l}\,x^{2l} \hspace{0.5 cm}\mbox{for }x\in [-1,1]
\]
where $b_{k\,,\,l}$ are integers defined recursively by
\[
\left\lbrace
\begin{array}{l}
b_{\,0\,,\,0}=1 \;\mbox{and}\;b_{k\,,\,0}\,=\,b_{\,0\,,\,l}=0\;\hspace{2.2 mm} \mbox{for}\;k,l\geq 1\, \\
b_{k+1\,,\,l+1}\,=\,b_{k\,,\,l}\,+\,l^2 b_{k+1\,,\,l}\;\hspace{6 mm}\mbox{for }k,l\geq 0\,. 
\end{array}
\right.
\]
\end{sublemma}
\begin{proof}
We note first that the functions $f_{2k}(x)\stackrel{def}{=}(\mbox{Arcsin}\, x)^{2k}$ satisfy
\[
(1-x^2)f''_{2k+2}(x)-x\,f'_{2k+2}(x)-(2k+2)(2k+1)f_{2k}(x)=0\hspace{0.5 cm}\mbox{for }x\in\; ]\!-1,1[.
\]
From the definition of $f_{2k}$ and the above equality it follows that numbers $c_{k\,,\,l}$ defined by
\[f_{2k}(x)= \sum_{l=0}^{\infty}\,c_{k\,,\,l}\,x^{2l} \hspace{0.5 cm}\mbox{for}\hspace{0.5cm} x\in [-1,1]
\]
are uniquely determined by the recurrence relations  
\[
\left\lbrace
\begin{array}{l}
 c_{\,0\,,\,0}\,=\,1 \;\mbox{and}\; c_{k\,,\,0}\,=\,c_{\,0\,,\,l}\,=\,0\hspace{3mm}\mbox{for }k,\,l\geq 1\,\\
(2l+2)(2l+1)c_{k+1\,,\,l+1}-4l^2c_{k+1\,,\,l}-(2k+2)(2k+1)c_{k\,,\,l}=0 \hspace{0.2 cm}\mbox{for }k,\,l\geq 0\,.
\end{array}
\right.
\]
A simple check shows that $c_{k \,,\,l}=\displaystyle \frac{(2k)!}{(2l)!}2^{2l-2k}b_{k\,,\,l}$.
\end{proof}
\begin{sublemma}\label{42}
Let $b_{k \,,\,l}$ be the numbers defined in Sublemma \ref{41}. Then
\begin{equation}\label{eq:limit}
 \lim_{l\to\infty}\frac{b_{k\,,\,l}}{((l-1)!)^2}=\frac{\pi^{2k-2}}{(2k-1)!} \hspace{3 mm}\mbox{for all }k \geq 1.
\end{equation}
\end{sublemma}
\begin{proof}
From the definition of numbers $b_{k\,,\,l}$ we infer that $b_{1\,,\,l}=((l-1)!)^2$ for $l\geq 1$. Thus relation (\ref{eq:limit}) is trivially true for $k=1$. We then assume $k\geq 2$. As $b_{j\,,\,1}=0$ for $j\geq 2$ the numbers $d_{j\,,\,l}$ defined for $j,\,l\geq 1$ by 
$\displaystyle d_{j\,,\,l}=\frac{b_{j\,,\,l}}{((l-1)!)^2}$ satisfy the recurrence relations
\[
\left\lbrace
\begin{array}{l}
 d_{j\,,\,1}\,=\,0 \;\mbox{and }d_{1\,,\,l}\,=1\hspace{3mm}\mbox{for } j\geq 2\;\mbox{and}\;l\geq 1, \\
d_{j+1\,,\,l+1}=\displaystyle\frac{1}{l^2}d_{j\,,\,l}+d_{j+1\,,\,l}\hspace{1 cm}
\mbox{for }j,\,l\geq 1\,.\\
\end{array}
\right.
\]
Using the fact that $d_{j-1\,,\,l}\,=\,0$ for $l=1,\cdots,j-2$ we get first for $j\geq 2$ the equality
\[
d_{j\,,\,n_j}=\sum_{n_{j-1}=j-1}^{n_j-1}\frac{1}{n_{j-1}^2}d_{j-1\,,\,n_{j-1}}
\]
which we iterate to obtain
\[
d_{k\,,\,l}=\sum_{n_{k-1}=k-1}^{l-1}\frac{1}{n_{k-1}^2}  \sum_{n_{k-2}=k-2}^{n_{k-1}-1}\frac{1}{n_{k-2}^2} \cdots 
\sum_{n_2=2}^{n_3-1}\frac{1}{n_{2}^2}\sum_{n_1=1}^{n_2-1}\frac{1}{n_{1}^2}\,.
\]
This leads to
\[
\lim_{l\to\infty}d_{k\,,\,l}=\sum_{n_{k-1}>n_{k-2}>\cdots>n_2>n_1>0}\prod_{j=1}^{k-1}\frac{1}{n_j^2}
\]
and we recognize in the right-hand side the number $\zeta (\{2\}_{(k-1)})$ whose value, given in [2], is equal to the right-hand side of (\ref{eq:limit})\;.
\end{proof}
\begin{proof}[ Proof of Lemma~{\rm\ref{4}}]
It suffices to prove that the numbers $e_{m\,,\,l}$ defined by 
\begin{equation}\label{eq:defin}
(-1)^{m+1}B_{2m}(\,\frac{1}{2}+\frac{1}{\pi}\mbox{Arcsin}\,x)=\sum_{l=0}^{\infty}e_{m\,,\,l}\,x^{2l}
\end{equation}
satisfy $e_{m,\,\,l}>0$ for all $m,l\in \mathbb{N^*}$.
Using Taylor's formula and the evenness of function $B_{2m}(\frac{1}{2}~+~\frac{t}{\pi})$ we have
\[
B_{2m}(\,\frac{1}{2}+\frac{t}{\pi}) = \sum_{k=0}^{m}\frac{1}{(2k)!}\,B_{2m}^{(2k)}(\frac{1}{2})\,(\frac{t}{\pi})^{2k}=
\sum_{k=0}^{m}\binom{2m}{2k}\,B_{2m-2k}(\frac{1}{2})\,(\frac{t}{\pi})^{2k}
\]
and the Taylor expansion of $\displaystyle (\mbox{Arcsin}\, x)^{2k}$ given in Sublemma \ref{41} leads to 
\begin{eqnarray*}
B_{2m}(\,\frac{1}{2}+\frac{1}{\pi}\mbox{Arcsin}\, x)&=&
\sum_{k=0}^{m}\left( \binom{2m}{2k} \,B_{2m-2k}(\frac{1}{2})\,\pi^{-2k}\sum_{l=0}^{\infty}
\frac{(2k)!}{(2l)!} 2^{2l-2k} b_{k\,,\,l}x^{2l}\right)\\&=&
\frac{(2m)!}{(2\pi )^{2m}}\sum_{k=0}^{m}\left(\frac{(2\pi )^{2m-2k}}{(2m-2k)!}B_{2m-2k}(\frac{1}{2})\sum_{l=0}^{\infty}\frac{2^{2l}}{(2l)!}b_{k\,,\,l}x^{2l}\right)\,.
\end{eqnarray*}
We then change the order of summation to get
\begin{equation}\label{eq:perm}
(-1)^{m+1}B_{2m}(\,\frac{1}{2}+\frac{1}{\pi}\mbox{Arcsin}\, x)=\frac{(2m)!}{(2\pi)^{2m}}\sum_{l=0}^{\infty}\frac{2^{2l}}{(2l)!}f_{m\,,\,l}x^{2l}
\end{equation}
where
\[
f_{m\,,\,l}=(-1)^{m+1}\sum_{k=0}^m \frac{(2\pi )^{2m-2k}}{(2m-2k)!}B_{2m-2k}(\frac{1}{2})\, b_{k\,,\,l}\,.
\]
We prove by recurrence over $m$ that $f_{m\,,\,l}>0$ for $m,\,l\geq 1$. To this end we set 
$\displaystyle g_{m,l}=\frac{f_{m,l}}{((l-1)!)^2}$ for $m,\,l\geq 1$ and 
since $b_{0\,,\,l}=0$ for $l\geq 1$ we have
\begin{eqnarray*}
g_{m+1\,,\,l+1}&=&\frac{(-1)^{m+2}}{(l!)^2}\sum_{k=1}^{m+1}\frac{(2\pi )^{2m+2-2k}}{(2m+2-2k)!}B_{2m+2-2k}(\frac{1}{2})\,b_{k\,,\,l+1}\\&=&
\frac{(-1)^{m+2}}{(l!)^2}\sum_{k=1}^{m+1}\frac{(2\pi )^{2m+2-2k}}{(2m+2-2k)!}B_{2m+2-2k}(\frac{1}{2})\,\Big(b_{k-1\,,\,l}+l^2\,b_{k\,,\,l}\Big)\\&=&
\frac{(-1)^{m+2}}{(l!)^2}\sum_{k=1}^{m+1}\frac{(2\pi )^{2m+2-2k}}{(2m+2-2k)!}B_{2m+2-2k}(\frac{1}{2})\,b_{k-1\,,\,l}+g_{m+1\,,\,l}\\&=&
-\frac{(-1)^{m+1}}{(l!)^2}\sum_{k=0}^{m}\frac{(2\pi )^{2m-2k}}{(2m-2k)!}B_{2m-2k}(\frac{1}{2})\,b_{k\,,\,l}+g_{m+1\,,\,l}\\&=&
-\frac{1}{l^2}g_{m\,,\,l}+g_{m+1\,,\,l}
\end{eqnarray*}
and this implies that
\[
g_{m+1\,,\,l+1}+\frac{1}{l^2}g_{m\,,\,l}=g_{m+1\,,\,l}\hspace{3mm}\mbox{for }l\geq 1.
\]
We have $\displaystyle g_{1\,,\,l}\,=\,f_{1\,,\,l}\,=1$ for all $l\geq 1$. Let us suppose that $g_{m\,,\,l}>0$ for all $l\geq 1$. Then $g_{m+1\,,\,l+1}<g_{m+1\,,\,l}$ and it follows that  
 $\displaystyle g_{m+1\,,\,l}>\lim_{l \to \infty}g_{m+1\,,\,l}$\,.
 Thanks to Sublemma \ref{42} we have
\begin{eqnarray}\label{eq:lim}
\lim_{l \to \infty}g_{m+1\,,\,l}&=&(-1)^{m+2}\sum_{k=1}^{m+1} \frac{(2\pi )^{2m+2-2k}}{(2m+2-2k)!}B_{2m+2-2k}(\frac{1}{2})\,\frac{\pi^{2k-2}}{(2k-1)!}\nonumber\\
&=&(-1)^{m+2}\pi^{2m}\sum_{k=1}^{m+1} \frac{2^{2m+2-2k}}{(2m+2-2k)!(2k-1)!}B_{2m+2-2k}(\frac{1}{2})
\end{eqnarray}
and using $\displaystyle B_j(\frac{1}{2})=0$ for all odd $j$ and the formula
\[
B_n(x+y)=\sum_{j=0}^n \binom{n}{j}B_j(x)y^{n-j}
\] 
we check that the sum which appears in (\ref{eq:lim}) is equal to 
\begin{eqnarray*}
\sum_{j=0}^{2m+1}\frac{2^j}{j!(2m+1-j)!}B_j(\frac{1}{2}) &=&
 \frac{2^{2m+1}}{(2m+1)!}\sum_{j=0}^{2m+1}\binom{2m+1}{j} B_j(\frac{1}{2})\left( \frac{1}{2}\right)^{2m+1-j}\\
 &=&\frac{2^{2m+1}}{(2m+1)!}B_{2m+1}(1)=0\,.
\end{eqnarray*}
Hence $g_{m\,,\,l}>0$ for $m,l\geq 1$ and this implies, thanks to (\ref{eq:perm}), that the numbers $e_{m\,,\,l}$ defined by (\ref{eq:defin}) are positive for $m,l\geq 1$.
\end{proof}
We are now in position to prove main properties of functions $\Psi_{2l-1}^*(\cdot,\ldots,\cdot,\pm a)$.
\begin{lemma}\label{5}
Let $\Psi_{2l-1}^*(\cdot,\ldots,\cdot,\pm a)$ be the functions defined in Lemma \ref{3}. Then 
\begin{enumerate}[\indent]
\item[a)] $\displaystyle(-1)^{n+l+1}\Psi_{2l-1}^*(x_{-n},\ldots,x_n,\pm a)>0$ for $(x_{-n},\ldots,x_n)\in \Omega$\,.
\item[b)] $\displaystyle(-1)^{n+l+1}\frac{\partial}{\partial x_i}\Psi_{2l-1}^*(x_{-n},\ldots,x_n,\pm a)$
 is negative if $i=-n,\ldots,-1$ and positive if $i~=~1,\ldots,n$ for $(x_{-n},\ldots,x_n)\in \Omega$.
 \footnote{By $\displaystyle \frac{\partial}{\partial x_i}\Psi_{2l-1}^*(\cdot,\ldots,\cdot,\pm a)$ we mean the continuous extension of $\displaystyle \frac{\partial}{\partial x_i}\Psi_{2l-1}(\cdot,\ldots,\cdot,\pm a)$ to $\Omega$. }
\item[c)]For $(x_{-n},\ldots,x_n)\in \omega$ and $c$ such that $0<c\,a<n\pi$ and $c\,a\neq j\,\pi$ where $j=1,\ldots,n-1$ we have the upper bound
\footnote{ It follows from the proof that the singularities of the right-hand side when $c\,a=j\,\pi$ for $j=1,\ldots,n-1$ are removable.}
\[
\hspace{-3pt} \sum_{k=1}^{m}(-1)^{n+k+1}\Psi_{2k-1}^*(x_{-n},\ldots,x_n,a)c^{2k-1}+\sum_{k=1}^{m}(-1)^{n+k+1}\Psi_{2k-1}^*(x_{-n},\ldots,x_n,-a)c^{2k-1}
 \]
\begin{equation}\label{eq:sumsimple}
\leq\Big{(}(-1)^n\displaystyle\prod_{\genfrac{}{}{0 cm}{2}{-n\leq j\leq n}{\; j\neq 0}}\sin (\pi\,\frac{x_j}{2a})\Big{)}\Big{(}\frac{-1}{\sin(c\,a)}\sum_{k=-n}^n \frac{\cos(c\,x_k)}{\displaystyle \prod_{\genfrac{}{}{0 cm}{2}{-n\leq j\leq n}{\; j\neq k}}\Big{(}\sin(\pi \frac{x_k}{2a})-\sin(\pi \frac{x_j}{2a})\Big{)}}\Big{)}\,.
\end{equation}
\end{enumerate}
\end{lemma}
\newpage
\begin{proof}
{~}
\begin{enumerate}[\indent]
\item[a)]
For $(x_{-n},\ldots,x_n)\in \omega$ we have
\[
\Psi_{2l-1}^*(x_{-n},\ldots,x_n,\pm a)= 2\frac{(4a)^{2l-1}}{(2l)!}\,\sum_{k=-n}^{n}\,\mu_k\,B_{2l}\big{(}\frac{1}{2}+\frac{\pm a+x_k}{4a}\big{)}
\]
since the function $\displaystyle B_{2m}\big{(}\frac{1}{2}+t \big{)}$ is even
and then
\[
(-1)^{n+l+1}\,\Psi_{2l-1}^*(x_{-n},\ldots,x_n,\pm a) 
\]
\[
=2\frac{(4a)^{2l-1}}{(2l)!}\left ( \frac{(-1)^n}{\alpha_0}\right )\sum_{k=-n}^n{\alpha_k(-1)^{l+1} B_{2l}\big{(}\frac{1}{2}+\frac{\pm a+x_k}{4a}\big{)}}\,.
\]
The first two terms of the right-hand side are positive and the third term writes 
$\displaystyle \sum_{k=-n}^n \alpha_k h(\sin(\pi \frac{x_k}{2a}))$ where 
\[
h(t)=(-1)^{l+1}B_{2l}(\frac{3}{4}\pm\frac{1}{2\pi}\mbox{Arcsin}\,t )\hspace{2 mm} \mbox{for} \hspace{2 mm} t\in[-1,1].
\]
The identity
\[
\frac{3}{4}\pm\frac{1}{2\pi}\mbox{Arcsin}\,t=\frac{1}{2}+
\frac{1}{\pi}\mbox{Arcsin}\,\sqrt{\frac{1\pm t}{2}}\hspace{2 mm} \mbox{for} \hspace{2 mm} t\in[-1,1]
\]
together with Lemma \ref{4} show that $h^{(2n)}$ is positive on $]-1,1[$ and the conclusion holds by Lemma~\ref{2}.
\item[b)] For $(x_{-n},\ldots,x_n)\in \omega$ and with the notations of a) we have 
\[
(-1)^{n+l+1}\frac{\partial}{\partial x_i}\Psi_{2l-1}^*(x_{-n},\ldots,x_n,\pm a)
\]
\[
=2\frac{(4a)^{2l-1}}{(2l)!}\left ((-1)^n \displaystyle \prod_{\genfrac{}{}{0 cm}{2}{-n\leq j\leq n}{ j\neq 0,i}}\sin ( \pi \frac{x_j}{2a})\right)\frac{\partial}{\partial x_i}\big{(}\sin ( \pi \frac{x_i}{2a})\sum_{k=-n}^n\alpha_k\, h(\sin(\pi \frac{x_k}{2a}))\big{)}.
\]
The first and the third term of the right-hand side are positive whereas the second is negative if $i\leq -1$ and positive if $i\geq 1$. The conclusion holds by Lemma~\ref{2}.
\item[c)]
The use of the Fourier series expansion
\[
B_{2l}(x)=(-1)^{l+1}2((2l)!)\sum_{j=1}^{\infty}\frac{1}{(2j\pi)^{2l}}\cos(2j\pi x)\hspace{0.5cm} \mbox{for }x\in[0,1]
\]
and the identity 
$\displaystyle \cos \alpha +\cos \beta=2\cos(\frac{\alpha +\beta}{2})\cos(\frac{\alpha -\beta}{2})$  lead to the expression 
\[
\Psi^*_{2l-1}(x_{-n},\ldots,x_n,x)
\] 
\[=(-1)^{l+1}\, 2\, \frac{(2a)^{2l-1}}{\alpha_0 \pi^{2l}}\,\sum_{j=1}^{\infty}\,\frac{1}{j^{2l}}\Big{(}\sum_{k=-n}^{n}\alpha_k\,\cos(j\pi(\frac{1}{2}+\frac{x_k}{2a}))\Big{)}\cos(j\pi(\frac{1}{2}+\frac{x}{2a}))\,.
\]
Using the identity $\displaystyle \cos(j\pi(\frac{1}{2}+y))=(-1)^jT_j(\sin(\pi y))$ and setting\\
 $\displaystyle a_{j,k}=(-1)^j\;T_j(\sin(\pi \frac{x_k}{2a}))$ we have
$\displaystyle \sum_{k=-n}^{n}\alpha_k\,a_{j,k}=0$ for $j=1,\ldots,2n-1$ and therefore 
\begin{equation}\label{eq:tcheby}
\Psi^*_{2l-1}(x_{-n},\ldots,x_n,x)=(-1)^{l+1}\, 2\, \frac{(2a)^{2l-1}}{\alpha_0 \pi^{2l}}\,\sum_{j=2n}^{\infty}\,\frac{1}{j^{2l}}\Big{(}\sum_{k=-n}^{n}\alpha_k\,a_{j,k}\Big{)}\cos(j\pi(\frac{1}{2}+\frac{x}{2a}))\,.
\end{equation}
It follows that
\[
\sum_{k=1}^{\infty}(-1)^{n+k+1}\Psi_{2k-1}^*(x_{-n},\ldots,x_n,\pm a)c^{2k-1}
\]
\[
=
\frac{(-1)^{n}}{\alpha_0}\, 2\sum_{p=2n}^{\infty}\,\sum_{j=-n}^n\alpha_j\,\,a_{p\,,\,j}\,(\mp 1)^p\,\frac{1}{2 a c}\sum_{k=1}^{\infty}\left(\frac{2 a c }{p\pi}\right)^{2k}
\]
and by a) we have 
\[
\hspace{- 1.23 pt} \sum_{k=1}^{m}(-1)^{n+k+1}\Psi_{2k-1}^*(x_{-n},\ldots,x_n,a)c^{2k-1}+\sum_{k=1}^{m}(-1)^{n+k+1}\Psi_{2k-1}^*(x_{-n},\ldots,x_n,-a)c^{2k-1}
 \]
\begin{eqnarray*}
&\leq&\frac{(-1)^{n}}{\alpha_0}\, 2\,\sum_{p=2n}^{\infty}\,\sum_{j=-n}^n\alpha_j\,a_{p\,,\,j}\,((-1)^p+1)\,\frac{2ac}{p^2\pi^2-4 a^2 c^2}\\
&=&\frac{(-1)^{n}}{\alpha_0}\, 2\,\sum_{q=n}^{\infty}\,\sum_{j=-n}^n\alpha_j\,a_{2q\,,\,j}\,\frac{ac}{q^2\pi^2-a^2 c^2}\\
&=&\frac{(-1)^{n}}{\alpha_0}\,2\,\sum_{q=1}^{\infty}\,\sum_{j=-n}^n\alpha_j\,a_{2q\,,\,j}\,\frac{ a c}{q^2\pi^2- a^2 c^2}\\
&=&\frac{(-1)^{n}}{\alpha_0}\, 2\,\sum_{q=1}^{\infty}\,\sum_{j=-n}^n \alpha_j\,\,\cos(2q\pi(\frac{1}{2}+\frac{x_j}{2a}))\,\frac{ac}{q^2\pi^2-a^2c^2}\\
&=&\frac{(-1)^{n+1}}{\alpha_0}\, 2\,\sum_{q=1}^{\infty}\,\sum_{j=-n}^n \alpha_j\,(-1)^{(q-1)}\,\cos(q\frac{\pi x_j}{a})\,\frac{ac}{q^2\pi^2-a^2c^2}\\
\end{eqnarray*}
where we use successively the absolute convergence to change the order of summation and the relations $\displaystyle \sum_{j=-n}^n\alpha_j a_{2q\,,\,j} =0$ for $q=1,\cdots,n-1$.\\ We finally make use of the identity ([2], formula (17.3.10))
\[
\sum_{k=1}^{\infty}\displaystyle\frac{(-1)^{k-1}}{k^2 r^2-s^2}\cos(kx)=\displaystyle\frac{\pi}{2rs}\displaystyle\frac{\cos(\frac{s}{r} x)}{\sin(\pi\frac{s}{r})}-\displaystyle\frac{1}{2s^2}\hspace{3 mm}\mbox{where}\hspace{3 mm}\frac{s}{r}\not\in \mathbb{Z}\hspace{3 mm}\mbox{and}\hspace{3 mm} -\pi\leq x\leq \pi
\]
to check that the last term of the above equalities is equal to the right-hand side of (\ref{eq:sumsimple}).
\end{enumerate}
\end{proof}
The next step is to bound the right-hand side of (\ref{eq:sumsimple}) for particular values~of~$x_k$.
\begin{lemma}\label{6}
For $\epsilon \in ]\,0,\log 2\,[$ and $c \in ]\,\frac{3}{4},1\,[$  we introduce $\delta=1-c$, $ \eta = (\,\log 2 - \epsilon\,)\frac{\delta}{\vert \log\delta \vert}$, $ a~=~(n-\frac{1}{2})\,\frac{\pi}{c}$ and $ s^*=\eta\,a$. Further let $x_0^*=0$ and $(x_{-n}^*,\ldots,x_n^*)\in \omega$ such that
\[
 \vert x_{\pm k}^* \vert = (k  -1)\pi +s^* \hspace{3mm}\mbox{for}\hspace{3mm} k=1,\ldots,n. 
 \]
 Then there exists  $0<c_{\epsilon}<1$ such that for $\displaystyle c_{\epsilon}<c<1-\frac{1}{2n}$ we have 
\begin{equation}\label{eq:prod}
0<(-1)^n\,\displaystyle\prod_{\genfrac{}{}{0 cm}{2}{-n\leq j\leq n}{\, j\neq 0}}\sin(\pi\,\frac{x_j^*}{2a})\,< 2^{-2n}
\end{equation}
and
\begin{equation}\label{eq:divided}
0<\frac{-1}{\sin(c\,a)}\sum_{k=-n}^n \frac{\cos(c\,x_k^*)}{\displaystyle \prod_{\genfrac{}{}{0 cm}{2}{-n\leq j\leq n}{\, j\neq k}}\Big{(}\sin(\pi \frac{x_k^*}{2a})-\sin(\pi \frac{x_j^*}{2a})\Big{)}}< 2^{2n-1}\,.
\end{equation}
\end{lemma}
The proof of (\ref{eq:divided}) requires preliminary results which we state in the next sublemmas.
\begin{sublemma}\label{63}
{~}
\begin{enumerate}[\indent]
\item[a)]Let $n\geq 2$ be even and let $0\leq t_0<t_1<\ldots<t_n\leq1$ and $0\leq t^*_0<t^*_1<\ldots<t^*_n\leq1$ such that
\[
\displaystyle t_{\frac{n}{2}+j}+t_{\frac{n}{2}-j}=1 \hspace{3 mm}\mbox{and}\hspace{3 mm}t^*_{\frac{n}{2}+j}+t^*_{\frac{n}{2}-j}=1\hspace{3 mm}\mbox{for}\hspace{3 mm}j=0,1,\ldots,\frac{n}{2}
\]
and
\[
t_{\frac{n}{2}+j}\leq t^*_{\frac{n}{2}+j}\hspace{3 mm}\mbox{for}\hspace{3 mm}j=1,2,\ldots,\frac{n}{2}.
\]
Let further $f\in C[0,1]\bigcap C^{n+2}[0,1[$ such that $f^{(n+2)}\geq 0$ on $[0,1[$.
Then
\[
\sum_{k=0}^{n}\frac{f(t_k)}{\displaystyle \prod_{\genfrac{}{}{0 cm}{2}{0\leq j\leq n}{\, j\neq k}}(t_k-t_j)}
\leq
\sum_{k=0}^{n}\frac{f(t^*_k)}{\displaystyle \prod_{\genfrac{}{}{0 cm}{2}{0\leq j\leq n}{\, j\neq k}}(t^*_k-t^*_j)}\,.
\]
\item[b)]Under the same assumptions on $f$, a similar result holds for odd integers $n$. \\
\end{enumerate}
\end{sublemma}
\begin{proof}
The proofs of a) and b) are similar so we only prove a). Let $g$ be the function defined for pairwise distinct $x_0,\ldots,x_n\in [0,1]$ by
\[
g(x_0,\ldots,x_{n})=\sum_{k=0}^{n}\frac{f(x_k)}{\displaystyle \prod_{\genfrac{}{}{0 cm}{} {0\leq j\leq n}{\; j\neq k}}(x_k-x_j)}\,.
\]
Since divided differences are invariant by permutation and since $f$ is continuous on $[0,1]$ it is sufficient to prove that
\[
\frac{\partial}{\partial t_n }g(\frac{1}{2},1-t_{\frac{n}{2}+1},\,t_{\frac{n}{2}+1},1-t_{\frac{n}{2}+2},\,t_{\frac{n}{2}+2},\ldots,1-t_{n-1},\,t_{n-1},1-t_n,\,t_n)\geq 0
\]
holds for pairwise distinct $t_{\frac{n}{2}+1},\ldots,t_n\in ]\frac{1}{2},1[$. We complete the proof using the representation formula and integrating by parts to get
\[
\ds \frac{\partial}{\partial t_n }g(\frac{1}{2},1-t_{\frac{n}{2}+1},\,t_{\frac{n}{2}+1},1-t_{\frac{n}{2}+2},\,t_{\frac{n}{2}+2},\ldots,1-t_{n-1},\,t_{n-1},1-t_n,\,t_n)=
\]
\[
\hspace{-2mm}\ds \displaystyle\int\limits_0^1\,d\tau_1\cdots\int\limits_0^{\tau_{n-1}}\,
f^{(n+1)}\,(\frac{1}{2}+\tau_1(\frac{1}{2}-t_{\frac{n}{2}+1})+\cdots+\tau_{n-1}(1-t_n-t_{n-1})+\tau_{n}(2t_n-1))(2\tau_n-\tau_{n-1})\, d\tau_{n}\hspace{2mm}=
\]
\[
\ds \displaystyle\int\limits_0^1\,d\tau_1\cdots\int\limits_0^{\tau_{n-1}}\,
f^{(n+2)}\,(\frac{1}{2}+\tau_1(\frac{1}{2}-t_{\frac{n}{2}+1})+\cdots+\tau_{n}(2t_n-1))(2t_n-1)\tau_n(\tau_{n-1}-\tau_n)\, d\tau_{n}\hspace{2mm}\geq 0\,.
\]
\end{proof}
\begin{sublemma}\label{64}
Let $n$ be a positive integer and $\displaystyle t^*_k=\sin^2(k\frac{\pi}{2n})$ for $k=0,1,\ldots,n$. Then
\begin{equation}\label{eq:comp}
\Big{\vert}\,\sum_{k=0}^{n}\frac{\cos((2n-1)\mbox{Arcsin}\sqrt{t^*_k}\,)}{\displaystyle \prod_{0\leq j\leq n\atop j\neq k}(t^*_k-t^*_j)}\,\Big{\vert}\leq 2^{2n-1}.
\end{equation}
\end{sublemma}
\begin{proof}
Setting $\displaystyle \gamma_k=\Big{(}\prod_{0\leq j\leq n\atop j\neq k}(t^*_k-t^*_j)\Big{)}^{-1}$ we have $\displaystyle \gamma_0=(-1)^n\frac{2^{2n-2}}{n}$ and $\displaystyle \gamma_n=(-1)^n \gamma_0$.\\Further for $k=1,2,\ldots,n-1$, using the identity $\sin^2 a-\sin^2 b=\sin(a-b)\sin(a+b)$, we have
\begin{eqnarray*}
\displaystyle \sin(\displaystyle k\frac{\pi}{n}) \frac{1}{\gamma_k}&=&\prod_{0\leq j\leq n\atop j\neq k}\sin((k-j)\frac{\pi}{2n})\prod_{0\leq j\leq n}\sin((k+j)\frac{\pi}{2n})\\
&=&(-1)^{n-k}\negmedspace\prod_{1\leq l\leq k}\negmedspace\sin(l\frac{\pi}{2n})\negmedspace\prod_{1\leq l\leq n-k}\negmedspace\sin(l\frac{\pi}{2n})\negmedspace\prod_{k\leq l\leq n}\sin(l\frac{\pi}{2n})\negmedspace\prod_{n-k < j\leq n}\negmedspace\sin(\pi-(k+j)\frac{\pi}{2n})\\
&=&\Big{(}(-1)^{n-k}\prod_{1\leq l\leq n}\sin^2(l\frac{\pi}{2n})\Big{)}\Big{(}\,\sin(\displaystyle k\frac{\pi}{2n})\,\sin(\displaystyle (n-k)\frac{\pi}{2n})\Big{)}= \sin(\displaystyle k\frac{\pi}{n})\frac{(-1)^k}{2\gamma_0}\\
\end{eqnarray*}
and therefore $\gamma_k=(-1)^k 2\gamma_0$. We bound trivially the right-hand side of~(\ref{eq:comp}) to complete the proof.
\end{proof}
\begin{sublemma}\label{65}
Let $n$ be a positive integer and $f(t)=\cos((2n-1)\mbox{Arcsin}\,\sqrt{t}\,)$. Then
\[
(-1)^n f^{(k)}>0\hspace{2 mm}\mbox{on}\hspace{2mm} [0,1[\hspace{2 mm}\mbox{for}\hspace{2mm} k=n,n+1,\ldots
\]
\end{sublemma}
\begin{proof}
We have $\displaystyle f(t)=F(n-\frac{1}{2},-n+\frac{1}{2},\frac{1}{2},t)$ where $F$ is Gauss' hypergeometric function [4] and this implies that $\displaystyle f(t)=\sum_{k=0}^{\infty}\frac{\alpha_k}{(2k)!}\,t^k$ where
\[
\alpha_k=\prod_{j=0}^{k-1}\,(4j^2-(2n-1)^2)\,.
\]
Hence $(-1)^n \alpha_k>0$ for $k\geq n$ and this leads to $(-1)^n f^{(k)}>0$ on $[0,1[$ for $k\geq n$.
\end{proof}
\begin{proof}[Proof of Lemma~{\rm\ref{6}}]
To prove (\ref{eq:prod}) we introduce the functions $\displaystyle g(t)=\log\,\sin(\frac{\pi}{2}t)$ and $\displaystyle G(t)=\int_{0}^t g(\tau)\,d\tau$. Since $g'\geq 0$ and $g''< 0$ on $]0,1]$ we have
\begin{eqnarray*}
\sum_{k=1}^n g(\frac{(k-1)\pi+s^*}{a})&\leq &\displaystyle \int\limits_0^{n-\frac{1}{2}} g(\frac{x\pi+s^*}{a})\,dx\,+\,\frac{1}{2}\;g(\,\frac{s^*}{a}\,)=\frac{a}{\pi}\int\limits_{\eta}^{1-\delta+\eta}g(t)\,dt\,+\,\frac{1}{2}\;g(\eta)\\
&=&\frac{a}{\pi}(G(1-\delta+\eta)-G(\eta))\,+\,\frac{1}{2}\;g(\eta)
\end{eqnarray*}
and therefore
\begin{eqnarray*}
\log \Big{(}(-1)^n \displaystyle \prod_{-n\leq k\leq n\atop k\neq 0}\sin(\frac{\pi}{2}\,\frac{x_k^*}{a})\Big{)}&\leq & 2\sum_{k=1}^n g(\frac{(k-1)\pi+s^*}{a}) 
 \leq \frac{2a}{\pi}\,(G(1-\delta+\eta)-G(\eta))\,+\,g(\eta)\\
& = &\frac{2n-1}{c}\,(G(1-\delta+\eta)-G(\eta))\,+g(\eta)\\
& = &\frac{2n-1}{c}\,(c\log 2+G(1-\delta+\eta)-G(\eta))\,-(2n-1)\log 2\,+g(\eta)\\
&=&\frac{2n-1}{c}\,\,h(\delta)\,-(2n-1)\log 2 +\,g(\eta)
\end{eqnarray*}
where
\[
h(\delta)=(1-\delta)\log2+G(1-\delta+(\,\log 2 -\epsilon\,)\frac{\delta}{\vert\log\delta\vert})-G((\,\log 2 -\epsilon\,)\frac{\delta}{\vert\log\delta\vert}).
\]
As $G(1)=-\log 2$ we have $\displaystyle \lim_{\delta \to 0_+}h(\delta)=0$ and $\displaystyle \lim_{\delta \to 0_+}h'(\delta)=-\epsilon$ and there exists $0<\delta_{\epsilon}< \frac{1}{4}$ such that $h<0$ on $]0,\delta_{\epsilon}[$. Moreover $ g(\eta)<g(\frac{1}{8})<\log\frac{1}{2}$ and the bound (\ref {eq:prod}) holds with $c_{\epsilon}=1-\delta_{\epsilon}$.\\
To check (\ref{eq:divided}) we set $t_k(s)=\sin^2(\pi\frac{(k-1)\pi+s}{2a})$ for $s>0$ and $k=1,2,\ldots,n$ and we introduce the function  
\[
\phi(y_0,\ldots,y_n)=(-1)^n\,\sum_{k=0}^n \frac{\cos((2n-1)\mbox{Arcsin}\sqrt{y_k})}{\displaystyle \prod_{0\leq j\leq n\atop j\neq k}\Big{(}y_k-y_j\Big{)}}
\]
defined for $0\leq y_0<y_1<\ldots<y_n\leq 1$.\\
Using Lemma \ref{2} and Sublemma \ref{65} we easily check that $\phi >0$ and $\phi$ is increasing in each argument.
Further $x^*_{-k}=-x^*_k$ and $\displaystyle cx^*_k=\frac{2ca}{\pi}\mbox{Arcsin}\,\sqrt{t_k(s^*)}=(2n-1)\mbox{Arcsin}\,\sqrt{t_k(s^*)}$ for $k~=~1,\ldots,n$ and we have 
\[
\frac{-1}{\sin(c\,a)}\,\sum_{k=-n}^{n}\frac{\cos(c\,x^*_k)}{\displaystyle \prod_{\genfrac{}{}{0 cm}{2}{-n\leq j\leq n}{\,j\neq k}}\Big{(}\sin(\pi \frac{x^*_k}{2a})-\sin(\pi \frac{x^*_j}{2a})\Big{)}}=
\frac{-1}{\sin(c\,a)}\,\sum_{k=0}^{n}\frac{\cos(c\,x^*_k)}{\displaystyle \prod_{\genfrac{}{}{0 cm}{2}{0\leq j\leq n}{\,j\neq k}}\Big{(}\sin^2(\pi \frac{x^*_k}{2a})-\sin^2(\pi \frac{x^*_j}{2a})\Big{)}}
\]
\[
=\phi(0,t_1(s^*),\ldots,t_n(s^*))\leq \phi(0,t_1(s^*),\ldots,t_{n-1}(s^*),1)
\]
since $\sin(ca)=(-1)^{n+1}$.\\
Assume that $n$ is even. Then
\[
t_{\frac{n}{2}-j}(s^*)+t_{\frac{n}{2}+j}(s^*)=1-\cos(\pi \frac{(\frac{n}{2}-1)\pi+s^*}{a})\cos(\pi \frac{j\pi}{a})<1\hspace{3mm}\mbox{for}\hspace{3mm}j=0,1,\ldots,\frac{n}{2}-1
\]
since
\[
\frac{(\frac{n}{2}-1)\pi+s^*}{a}<\frac{c}{2}+\frac{s^*}{a}<\frac{1}{2}-\frac{\delta}{2}+\log 2 \,\frac{\delta}{\vert\log \delta\vert}=\frac{1}{2}+\frac{\delta}{2} (\frac{2\log 2}{\vert\log \delta\vert
} -1)<\frac{1}{2}
\]
as $0<\delta<\frac{1}{4} $. Now we choose $s^{**}$ such that 
\[
\frac{(\frac{n}{2}-1)\pi+s^{**}}{a}=\frac{1}{2}
\]
and since $s^*<s^{**}$ and $t_j(s^*)<t_j(s^{**})$ we get  
\begin{equation}\label{eq:phi}
 \phi(0,t_1(s^*),\ldots,t_{n-1}(s^*),1)<  \phi(0,t_1(s^{**}),\ldots,t_{n-1}(s^{**}),1).
\end{equation}
As $ c<1-\frac{1}{2n}$ we have $a>n\pi$ and $ t_{\frac{n}{2}+j}(s^{**})=\sin^2(\pi(\frac{1}{4}+j\frac{\pi}{2a}))<t^*_{\frac{n}{2}+j}$ for $j=1,2,\ldots,\frac{n}{2}-1$
where $ t^*_k=\sin^2(k\frac{\pi}{2n})$. We complete the proof using Sublemmas \ref{63} and \ref{64} to bound the right-hand side of (\ref{eq:phi}).
If $n$ is an odd integer a similar proof holds.
  \end{proof}
The last point is to bound the integral which appears in the right-hand side of the identity (\ref{eq:main}). This is the content of Lemma \ref{7}.
\begin{sublemma}\label{71}
Let $b_{n,l}$ the numbers defined for integers $n\ge 1$ and $l\ge 0$ by 
\[
b_{n,l}=\Big{(}\frac{2n}{2n+l}\Big{)}^{2 n\log n-1 } \binom{4n+l-1}{ l}\,.
\]
Then there exists a constant $C$ such that $\displaystyle \sum_{l=0}^{\infty}b_{n,l}\leq C $ for all $n\ge 10$.\\\\
\end{sublemma}
\begin{proof}
Since
\[
b_{n,l}=\Big{(}\frac{2n}{2n+l}\Big{)}^{2 n\log n -1} \frac{\Gamma(4n+l)}{\Gamma(l+1)\Gamma(4n)}
\]
we have, using Stirling's formula
\[
b_{n,l}=O_n(l^{-2n\log n+4n})=O_n(l^{-2})
\]
for $n\geq 10$
and there exists a constant $C$ such that 
$\displaystyle \sum_{l=0}^{\infty}b_{10,l}\leq C$. We now show that $b_{n,l}\leq b_{10,l}$ for $n\geq 10$.
We have $\log b_{n,l}=g(n,l)$ where the function g is defined for $(x,y)\in [1,\infty[\times[0,\infty[ $  by
\[
g(x,y)=(2x\log x-1)\log(\frac{2x}{2x+y})+\log \Gamma(4x+y)-\log\Gamma(y+1)-\log\Gamma(4x).
\]
Straightforward computations lead to
\[
\frac{\partial g}{\partial x} (x,y)=(2\log x +2)\log(\frac{2x}{2x+y})+(2x\log x -1)\frac{y}{x(2x+y)}+4\Psi(4x+y)-4\Psi(4x)
\]
and
\[
\frac{\partial^2 g}{\partial y \partial x}(x,y)=-\frac{2(1+2x+y+y\log x)}{(2x+y)^2}+4\Psi'(4x+y)
\]
where $\Psi$ is the derivative of $\log \Gamma$. From now on we assume that $(x,y)\in [10,\infty[\times[0,\infty[$. We have $\displaystyle \frac{\partial g}{\partial x} (x,0)=0$ and since $\displaystyle\Psi'(z)=\sum_{k=0}^{\infty}\frac{1}{(z+k)^2}$ we get $\displaystyle \Psi'(z)\leq \frac{1}{z}+\frac{1}{z^2}$ for $z>0$ and therefore
\begin{eqnarray*}
\frac{\partial^2 g}{\partial y \partial x}(x,y)&\leq &-\frac{2(1+2x+y+y\log x)}{(2x+y)^2}+\frac{4}{4x+y}+\frac{4}{(4x+y)^2}\\
& = &\frac{2(-8x^2+y^2+8x^2y+6xy^2+y^3)}{(2x+y)^2(4x+y)^2}-\frac{2y\log x}{(2x+y)^2}\leq
\frac{2y(1-\log x)}{(2x+y)^2}\leq 0\,.
\end{eqnarray*}
Hence $ \frac{\partial g}{\partial x} (x,y)\leq  \frac{\partial g}{\partial x} (x,0)=0$ and this implies that $g(x,y)\leq g(10,y)$ and hence $b_{n,l}\leq b_{10,l}$ for $n\geq 10$.\\
\end{proof}
\vspace{3mm}
\begin{lemma}\label{7}
Let $n$ and $m$ be integers such that $n\geq 10$ and $m\geq n\log n $ and let further $\Psi_{2l-1}^*$ be the function defined in Lemma \ref{3}. Then there exists a constant $C$ such that
\[
\Big{\vert}\Psi_{2m-1}^*(x_{-n},\ldots,x_n,x)\Big{\vert}\leq \;\frac{2^{2n-1}}{\alpha_0\,a} \Big{(}\frac{a}{n\pi}\Big{)}^{2m}C
\]
for $(x_{-n},\ldots,x_n,x)\in \Omega \times [-a,a]$.
\end{lemma}
\begin{proof}
For $(x_{-n},\ldots,x_n,x)\in \omega \times [-a,a]$ we use relation (\ref{eq:tcheby}) and Lemma \ref{2} to get 
\[ 
\Big{\vert}\Psi_{2m-1}^*(x_{-n},\ldots,x_n,x) \Big{\vert}\leq\, 2\, \frac{(2a)^{2m-1}}{\alpha_0 \pi^{2m}}\,\sum_{j=2n}^{\infty}\,\frac{1}{j^{2m}}\;\frac{\Big{\vert}T_j^{(2n)}(\tau_j)\Big{\vert}}{(2n)!}
\]
for some $\tau_j \in ]\negmedspace-1,1[$.
It is well known that [8]
\[
\displaystyle \max_{-1\leq x\leq 1} \Big{\vert}T_j^{(2n)}(x)\Big{\vert}= T_j^{(2n)}(1)=2^{2n-1}\,(2n-1)!\,j \binom{2n+j-1}{ j-2n}
\]
for $j=2n,2n+1,\ldots$ and therefore
\[
\Big{\vert}\Psi_{2m-1}^*(x_{-n},\ldots,x_n,x) \Big{\vert}\leq \frac{2^{2n-1}}{\alpha_0\,a} \Big{(}\frac{a}{n\pi}\Big{)}^{2m}\sum_{j=2n}^{\infty}\Big{(}\frac{2n}{j}\Big{)}^{2m-1}\binom{2n+j-1}{ j-2n}.
\]
We set $j=2n+l$ and since $m\geq n\log n$  we have
\begin{eqnarray*}
\Big{\vert}\Psi_{2m-1}^*(x_{-n},\ldots,x_n,x) \Big{\vert}&\leq& \frac{2^{2n-1}}{\alpha_0\,a} \Big{(}\frac{a}{n\pi}\Big{)}^{2m}\sum_{l=0}^{\infty}\Big{(}\frac{2n}{2n+l}\Big{)}^{2m-1}\binom{4n+l-1}{l}\\
&\leq&\frac{2^{2n-1}}{\alpha_0\,a} \Big{(}\frac{a}{n\pi}\Big{)}^{2m}\sum_{l=0}^{\infty}b_{n,l}\,\leq \,\frac{2^{2n-1}}{\alpha_0\,a} \Big{(}\frac{a}{n\pi}\Big{)}^{2m}C
\end{eqnarray*}
thanks to Sublemma \ref{71}.
\end{proof}
\section{Proofs of Theorems and Conclusion} 
\begin{proof}[Proof of Theorem \ref{mainth}]
Assume that Theorem \ref{mainth} is not true. Then there exist $(2\log2)^{-1}<C<1$ and arbitrary large $T$ such that
\[
\max_{t\in [T-2\pi,T+2\pi]}\vert Z^{(k)}(t)\vert < \,\Big{(}1-\frac{\log\log\log T}{\log\log T}\Big{)}^k \Big{(}\log \sqrt{\frac{T}{2\pi}}\,\Big{)}^k\,\vert Z(T)\vert
\]
for $k=1,3,5,\ldots,2m-1,\,2m$ where $m=\lfloor C \log T \log\log T\rfloor$. For such a $T$ we have $Z(T)\neq 0$ and we introduce the function $f$ defined~by
 \[ 
 \displaystyle f(t) = \frac{Z(T+\displaystyle \frac{t}{\log \sqrt{\frac{T}{2\pi}}})}{Z(T)}\,.
 \]
It satisfies $f(0)=1$ and its zeros $x_{\pm k}$ numbered such that $ \ldots \leq x_{-1}<0<x_1\leq\ldots$ are such that
 \[
 \mid x_{\pm k}\mid \leq (k-1)\,\pi+\Big{(}\frac{\pi}{2}+o(1)\Big{)}\frac{\log T}{\log\log T}\hspace{3mm}\mbox{for }k=1,2,\cdots,\lfloor \sqrt{T}\rfloor.
 \]
 One easily checks that function $f$ satisfies the assumptions of Theorem \ref{genth} for large $T$  with
\[
n= \lfloor C \log T\rfloor\,,\hspace{5mm} c= 1-\frac{\log\log\log T}{\log\log T}\,\hspace{5 mm}s= \Big{(}\frac{\pi}{2}+o(1)\Big{)}\frac{\log T}{\log\log T}
\]
and relation (\ref{eq:lower}) leads to
\[
\Big{(}\frac{\pi}{2}+o(1)\Big{)}\frac{\log T}{\log\log T}\geq (\log 2 -\epsilon)\frac{\lfloor C\log T\rfloor}{\log\log T}\pi
\]
which is a contradiction for $\epsilon$ sufficiently small and $T$ sufficiently large.
\end{proof}
\begin{proof}[Proof of Theorem \ref{genth}]
Let $f$ be a function satisfying the assumptions of Theorem \ref{genth} for the parameters $n$, $c$ and $s$. We fix $0<\epsilon<\log 2$ and assume that $\displaystyle c_{\epsilon}<c<1-\frac{1}{2n}$ where $c_{\epsilon}$ is defined in Lemma \ref{6}.\\
We set $a=(n-\frac{1}{2})\frac{\pi}{c}$ and we suppose that $s<s^*$ where
\[
s^*=(\log 2 -\epsilon)\frac{1-c}{\vert\log(1-c)\vert}\,a
\]
to get a contradiction. 
Introducing $x_k^*= \mbox{sgn}(k)((\vert k\vert -1)\pi+s^*)$ with $\mbox {sgn}(0)=0$ and using identity (\ref{eq:main}) and Lemmas \ref{6} and \ref{7}  we have
\begin{eqnarray*}
f(0)&\leq &\sum_{k=1}^{m}(-1)^{n+k+1}\Psi_{2k-1}^*(x_{-n},\ldots,x_n,a)c^{2k-1}+\sum_{k=1}^{m}(-1)^{n+k+1}\Psi_{2k-1}^*(x_{-n},\ldots,x_n,-a)c^{2k-1}\\
&+&\Big{\vert}\int\limits_{-a}^a f^{(2m)}(x)\Psi_{2m-1}^*(x_{-n},\ldots,x_n,x)\,dx\,\Big{\vert}\\
&\leq &\sum_{k=1}^{m}(-1)^{n+k+1}\Psi_{2k-1}^*(x_{-n}^*,\ldots,x_n^*,a)c^{2k-1}+\sum_{k=1}^{m}(-1)^{n+k+1}\Psi_{2k-1}^*(x_{-n}^*,\ldots,x_n^*,-a)c^{2k-1}\\
&+&2^{2n}\Big{\vert}\displaystyle\prod_{-n\leq j\leq n\atop j\neq 0}\sin(\pi\,\frac{x_j^*}{2a})\Big{\vert}\Big{(}1-\frac{1}{2n}\Big{)}^{2n\log n}C\\
&<&\frac{1}{2}+\Big{(}1-\frac{1}{2n}\Big{)}^{2n\log n}C<1
\end{eqnarray*}
for $n$ sufficiently large. This is a contradiction since $f(0)=1$ and therefore
\[
s\geq s^*\geq (\log 2 -\epsilon)\frac{1-c}{\vert\log(1-c)\vert}\,n\,\pi
\]
as $a>n\pi$ since $c<1-\frac{1}{2n}$.
\end{proof}
The bound given in Lemma \ref{7} does not take into account the repartition of $x_k$ and Theorem~\ref{genth} should hold under the weaker assumption $m \geq n$. It should also be possible to use deep properties of the argument of the zeta function to get a stronger version of Theorem~\ref{mainth}.\\\\ 
 {\bf Acknowledgements} \\
 I am grateful to Professor Jean Descloux for his encouragements and to my colleague, Jean-François Hêche, for his valuable comments.

\end{document}